\documentclass[12pt]{article}
\usepackage[centertags]{amsmath}
\usepackage{amsmath}
\usepackage{amsthm}
\usepackage{color} %
\usepackage[francais]{babel}
\usepackage[latin1]{inputenc}
\usepackage[T1]{fontenc}
\usepackage{graphicx} %
\usepackage{amsfonts}
\usepackage{amssymb}
\newtheorem{theorem}{Theorem}[section]

\newtheorem{proposition}[theorem]{Proposition}
\theoremstyle{definition}
\newtheorem{definition}[theorem]{Definition}
\theoremstyle{remarque}
\newtheorem{remark}[theorem]{Remark}
\theoremstyle{example}
\newtheorem{example}[theorem]{Example}
\numberwithin{equation}{section}

\begin{document}
\begin{center}\Large\textbf{On fibering compact manifold over the circle}
\end{center}
\begin{center}
 {\bf  Ameth \, Ndiaye}\footnote{\tiny\bf{
 Universit\'e\, Cheikh \,Anta \,Diop, \,Dakar/\,D\'epartement de Math\'ematiques(FASTEF)\\Email:\,ameth1.ndiaye@ucad.edu.sn\\
 }}
\end{center}

\\

\begin{abstract}
In this paper, we show that any compact manifold that carries a $SL(n,\mathbb{R})$-foliation is fibered on the circle $S^1$.
\end{abstract}

\section{Introduction}

\begin{definition}
A codimension $n$ \textbf{foliation} $\mathcal{F}$ on a $(n+m)$-manifold $M$ is given by an open cover $\{U_i\}_{i\in I}$ and submersions
$f_i:U_i\longrightarrow T$ over an $n$-dimensional manifold $T$ and, for $U_i\cap U_j\neq\emptyset$,
a diffeomorphism $\gamma_{ij}:f_i(U_i\cap U_j)\longrightarrow f_j(U_i\cap U_j)$ such that $f_j=\gamma_{ij}\circ f_i$.
\end{definition}

 We say that $\{U_i,f_i,T,\gamma_{ij}\}$ is a foliated cocycle defining $\mathcal{F}$.\\
A \textbf{transverse structure} to $\mathcal{F}$ is a geometric structure on $T$ invariant by the local diffeomorphisms
$\gamma_{ij}$. We say that $\mathcal{F}$ is a \textbf{Lie $G$-foliation}, if $T$ is a Lie group $G$ and $\gamma_{ij}$
are restrictions of left translations on $G$. Such foliation can also be defined by a $1$-form $\omega$ on
$M$ with values in the Lie algebra $\mathcal{G}$ such that:\\
i) $\omega_x:T_xM\longrightarrow\mathcal{G}$ is surjective for every $x\in M$,\\
ii) $d\omega+\frac{1}{2}[\omega,\omega]=0$.\\

If $\mathcal{G}$ is Abelian, $\omega$ is given by $n$ linearly independant closed scalar $1$-forms $\omega_1,...\omega_n$.\\

In the general case, the structure of a Lie foliation on a compact manifold, is given by the following theorem due to E. F\'edida \cite{2}:\\

Let $\mathcal{F}$ be a Lie $G$-foliation on a compact manifold $M$. Let $\widetilde{M}$ be the universal covering of $M$ and  $\widetilde{\mathcal{F}}$ the lift of $\mathcal{F}$ to $\widetilde{M}$. Then there exist a homomorphism $h : \pi_1(M)\longrightarrow G$ and a locally trivial fibration $D:\widetilde{M}\longrightarrow G$ whose fibres are the leaves of $\widetilde{\mathcal{F}}$ and such that, for every $\gamma\in\pi_1(M)$, the following
diagram is commutative:\\

where the first line denotes the deck transformation of $\gamma\in\pi_1(M)$ on $\widetilde{M}$.\\

The group $\Gamma=h(\pi_1(M))$ (which is a subgroup of $G$) is called the holonomy group of $\mathcal{F}$ although the holonomy of each leaf is trivial. The fibration : $D:\widetilde{M}\longrightarrow G$ is called the developing map of $\mathcal{F}$.

\section{Lie-foliation with transverse group $SL(n,\mathbb{R})$}
In the first we want to decompose the group $SL(n,\mathbb{R})$ such that the group $GA$ is one of the factor.\\

Let us compute an explicite example. Let $GA$ be the Lie group of affine transformations $x\in\mathbb{R}\longmapsto ax+b\in\mathbb{R}$,
where $b\in\mathbb{R}$ and $a\in]0;+\infty[$.
It can be embedded in the group $SL(2,\mathbb{R})$ as follows:
$$\left(
    \begin{array}{c}
      x\longmapsto ax+b \\
      a>O \\
    \end{array}
  \right)
  \in GA\longmapsto \frac{1}{\sqrt{a}}\left(
                                                                \begin{array}{cc}
                                                                  a & b \\
                                                                  0 & 1 \\
                                                                \end{array}
                                                              \right)\in SL(2,\mathbb{R}$$
There exist \cite{4} a manifold $M$ equipped with a Lie $SL(2,\mathbb{R})$-foliation $\mathcal{F}$
with $GA$ as the closure of its holonomy group.
Then, the basic cohomology of $\mathcal{F}$ is the cohomology of differential forms on $SL(2,\mathbb{R})$ invariant by $GA$.
The quotient $SL(2;\mathbb{R})/GA$ is diffeomorphic to
the circle $S^1$.\\

\begin{example}
Let $\mathcal{F}_o$ the Lie $GA$-foliation on a compact manifold $M_o$. We suppose that $\mathcal{F}_o$ can't be obtain by inverse image of a homogenous foliation.
The projection $p:SL(2,\mathbb{R}\longrightarrow S^1$ have a section given by the decomposition $SL(2,\mathbb{R})\cong GA\times S^1$. Let $D_o:\widetilde{V}_o\longrightarrow GA$ the
the developing map of $\mathcal{F}_o$.\\

 The map $D:\widetilde{V}_o\times S^1\longrightarrow SL(2,\mathbb{R})$,$(\widetilde{x},y)\longmapsto D_o\widetilde{x}.\sigma(y)$
 is local trivial fibration, these fibers define a Lie $SL(2,\mathbb{R})$-foliation on $\widetilde{V}_o\times S^1$ and induce a Lie $SL(2,\mathbb{R})$-foliation $\mathcal{F}$ on the manifold
 $V=\widetilde{V}_o\times S^1$ which is not conjugate to a homogenous foliation.
 \end{example}
 \begin{theorem}[Tischler]
 If there exists on a compact manifold $M$ a closed differential form without singularities, then $M$ is fibered on the circle.
 \end{theorem}
\begin{proposition}[H.Dathe]
A compact manifold that carry a Lie $SL(2,\mathbb{R})$-foliation is fiber on the circle $S^1$.
\end{proposition}
Our aim is to generalize this proposition to Lie $SL(n,\mathbb{R})$-foliation. Before that we have
\begin{proposition}
We have the decomposition $SL(n,\mathbb{R})\cong T^{n-1}\times G$ where $T^{n-1}$ is the maximal tore of $SL(n,\mathbb{R})$ identify by the subgroup of diagonal matrices.
And $G$ is Lie group such that $Lie(G)=\bigoplus_{i\neq j}<E_{ij}>$ and $dim G=n^2-n$. Moreover we have:
$$[E_{ij},E_{kl}]=0\,\,if\,\,i\neq l\,\,and\,\, j\neq k$$

$$[E_{ij},E_{jl}]=E_{il}\,\,if\,\, i\neq l$$

$$[E_{ij},E_{ki}]=-E_{kj}\,\,if\,\, k\neq j$$

$$[E_{ij},E_{ji}]=E_{ii} - E_{jj}$$
\end{proposition}
\begin{proof}
Let $SL(n,\mathbb{R})$ the real special linear group and $T$ the maximal tore of $SL(n,\mathbb{R})$ identify with the subgroup of diagonal matrices. We denote by $X(T)$ the
group of morphism $T\longrightarrow \mathbb{R^{\times}}$.\\
Then $T$ act on the Lie algebra $\mathcal{G}$ of $SL(n,\mathbb{R})$ by conjugaison and we have
$$\mathcal{G}=\mathcal{H}\oplus\bigoplus_{i\neq j}\lambda E_{ij},$$ $\lambda\in\mathbb{R}$, where $\mathcal{H}$ is the subset of diagonal matrices of vanishing trace.\\
Using this splitting we then can decompose $SL(n,\mathbb{R})\cong T\times G$ such that $Lie T=\mathcal{H}$ and $Lie(G)=\bigoplus_{i\neq j}<E_{ij}>$,
$E_{ij}$ being the $n\times n$ matrix where the element in the ith line and the jth culumb is equal to 1 and the other elements are vanish.\\
Let $Y=(a_{ij}),i,j=1,...,n$ in $\mathcal{H}$, thus we have
$$\sum_i^na_{ii}=0\Rightarrow a_{11}=-\sum_{i\neq 1}a_{ii}$$ then
$$Y=\sum_{i\neq 1}a_{ii}Y_i$$ where $Y_i=(b_{kl})$ is the matrix with $b_{11}=-1$, $b_{kk}=1$, $k\neq 1$ and $b_{kl}=0$ for $k\neq l$.

We can easily note that the $(Y_i) , i=2,...,n$ are also linearily independant, so $\mathcal{H} = < Y_i, i=2,...,n> $ and then $dim\mathcal{H}=n-1$.
This imply than $dim T=n-1$, therefore $dim G=(n^2-1)-(n-1)=n^2-n$.\\
And by a simple calculation using matrix product, we have the value of the Lie bracket $[E_{ij},E_{kl}]$, this finish then the proof.
\end{proof}
\begin{remark}
The group $G$ in the previous proposition can be identify with the group $SO(n)\times SO(n)$ and the maximal tore $T$ is isomorphic to $\mathbb{R}^{\frac{n(n+1)}{2}-1}/SO(n)$ and then we have
$$SL(n,\mathbb{R})\cong SO(n)\times \mathbb{R}^{\frac{n(n+1)}{2}-1}$$
\end{remark}
\begin{proposition}
Let $\mathcal{F}$ be a Lie $G$-foliation on a compact manifold $M$, with $G=G_1\times G_2$. There exists a Lie
$G_i$-foliation $\mathcal{F}_i$ on $M$ induced by the foliation $\mathcal{F}$.
\end{proposition}

\begin{proof}
Let $G_1$ and $G_2$ be two Lie groups and $\mathcal{F}$ be a Lie $G$-foliation on a compact manifold $M$, where $G=G_1\times G_2$.\\ 
If $D$ is the developing map of $\mathcal{F}$ on the universal cover $\widetilde{M}$ of $M$, then the simple foliation
defined by $p_i\circ D$ (where $p_i$, $i=1,2$ is the projection of $G$ on $G_i$, $i=1,2$), pass in quotient and induces a foliation $\mathcal{F}_i$ on $M$.
\end{proof}

\begin{theorem}
A compact manifold that carry a Lie $SL(n,\mathbb{R})$-foliation fiber on the circle $S^1$.
\end{theorem}

\begin{proof}
Let $M$ be a compact manifold with a Lie $SL(n,\mathbb{R})$-foliation. We have also
$$SL(n,\mathbb{R})\cong SO(n)\times \mathbb{R}^{\frac{n(n+1)}{2}-1}$$
$$SL(n,\mathbb{R})\cong SO(n)\times \mathbb{R}^{\frac{n(n+1)}{2}-3}\times\mathbb{R}^{2}$$
 Now we take $G=SL(n,\mathbb{R})$, $G_1=SO(n)\times \mathbb{R}^{\frac{n(n+1)}{2}-3}$ and $G_2=\mathbb{R}^2$\\
 so using the proposition, the Lie $SL(n,\mathbb{R})$-foliation induces a Lie
 $\mathbb{R}^2$-foliation on $M$. Since $\mathbb{R}^2$ is abelian the structures equations of the Lie $\mathbb{R}^2$-foliation are closed 1-forms on $M$, then using the Tischler theorem, $M$ is a fibration over the circle.
 \end{proof}


\begin{thebibliography}{9}
\bibitem{1} D. Tischler, \emph{On fibering certain manifold over the circle, Topology 9 (1970), 153-154}.
\bibitem{2} E. Fedida, \emph{ Feuilletages du plan, feuilletage de Lie, th\'ese universit\'e Louis Pasteur, Strasbourg (1973)}.
\bibitem{3} E. Fedida, \emph{Sur les feuilletages de Lie, C. R. Acad. Sci. Paris 272 (1971)999 1001}
\bibitem{4} A. El Kacimi Alaoui, G. Guasp, M. Nicolau, \emph{On deformations of transversely homogeneous foliations, Topology 40 (2001), 1363-1393}.
\bibitem{5} H. Dathe, \emph{Sur l'existence des feuilletages de Lie, Th\'ese de troisi\'eme cycle, Universit\'e Cheikh Anta Diop, Dakar (S\'en\'egal) (1999)}.
\bibitem{6} S. Riche, \emph{Sur les repr\'esentations des groupes alg\'ebriques et des groupes quantiques}
 \end{thebibliography}
\end{document}